\documentclass[12pt,draft]{scrartcl}

\usepackage[]{amsmath, amssymb,amsfonts,amsthm,mathtools}
\usepackage{braket}
\usepackage{enumitem}
\usepackage{comment}
\usepackage{color}
\usepackage{multicol}

\makeatletter
\DeclareOldFontCommand{\rm}{\normalfont\rmfamily}{\mathrm}
\DeclareOldFontCommand{\sf}{\normalfont\sffamily}{\mathsf}
\DeclareOldFontCommand{\tt}{\normalfont\ttfamily}{\mathtt}
\DeclareOldFontCommand{\bf}{\normalfont\bfseries}{\mathbf}
\DeclareOldFontCommand{\it}{\normalfont\itshape}{\mathit}
\DeclareOldFontCommand{\sl}{\normalfont\slshape}{\@nomath\sl}
\DeclareOldFontCommand{\sc}{\normalfont\scshape}{\@nomath\sc}
\makeatother

\theoremstyle{plain}
\newtheorem{theorem}{Theorem}[section]
\newtheorem{lemma}[theorem]{Lemma}
\newtheorem{proposition}[theorem]{Proposition}
\newtheorem{corollary}[theorem]{Corollary}

\theoremstyle{definition}
\newtheorem{definition}[theorem]{Definition}

\newtheorem{remark}[theorem]{Remark}

\DeclareMathOperator{\Tr}{Tr}

\title {A Generalization of APN Functions for Odd Characteristic}
\author{
Masamichi Kuroda, 
Shuhei Tsujie}
\date{}

\begin{document}
\maketitle
	
\begin{abstract}
Almost perfect nonlinear (APN) functions on finite fields of characteristic two have been studied by many researchers. 
Such functions have useful properties and applications in cryptography, 
finite geometries and so on. 
However APN functions on finite fields of odd characteristic do not satisfy desired properties. 
In this paper, we modify the definition of APN function in the case of odd characteristic, and study its properties.  
\end{abstract}



\section{Introduction}
Let $ F = \mathbb{F}_{p^{n}} $ be a finite field of characteristic $ p $. 
A function $ f \colon F \rightarrow F $ is called \textbf{almost perfect nonlinear} (APN) if the equation 
\begin{align*}
D_{a}f(x) \coloneqq f (x + a) - f (x) = b
\end{align*}
has at most two solutions $x$ in $ F $ for all $ a \in F^{\times} $ and $ b \in F $. 
APN functions on a finite field of characteristic $ 2 $ were introduced by Nyberg \cite{Nyberg1994differentially} and have been studied by many researchers. 
There are a lot of applications in cryptography and finite geometry. 
APN functions for odd characteristic have been investigated by \cite{DMMPW2003apn,HRS1999new} but their algebraic properties is quite different from the case of characteristic $ 2 $. 
In this paper, we give an algebraic generalization of APN functions as follows: 
\begin{definition}
A function $ f \colon F \rightarrow F $ is a \textbf{generalized almost perfect nonlinear}  (GAPN) function if the equation 
\begin{align*}
\tilde{D}_{a}f(x) \coloneqq \sum_{i \in \mathbb{F}_{p}}f(x+ia) = b
\end{align*}
has at most $ p $ solutions $ x $ in $ F $ for all $ a \in F^{\times} $ and $ b \in F $. 
\end{definition}
Note that when $ p=2 $ GAPN functions coincide with APN functions. 
For every $ a, b \in F $ the number of solutions 
\begin{align*}
\tilde{N}_{f}(a,b) \coloneqq \# \Set{x \in F | \tilde{D}_{a}f(x)=b}
\end{align*}
is divisible by $ p $ since if $ x $ is a solution then each element in $ x+\mathbb{F}_{p}a $ is also a solution.  
Therefore we have that 
$f$ is a GAPN function if and only if $\tilde{N}_{f}(a,b)$ equals zero or $p$ for any $a \in F^{\times}$ and $b \in F$. 
If $f$ is linear, for any $a \in F$ we have 
\begin{align*}
\tilde{N}_f (a, b) = \left\{ 
\begin{array}{cl}
0 & (b \not = 0 \ \mbox{and} \ p > 2), \\
p^n & (b = 0 \ \mbox{or} \ (p, b) = (2, f (a))). 
\end{array}
\right.
\end{align*}
Hence we may say that GAPN functions are the farthest from linear functions in view of this parameter. 

Our main resuts are following three theorems 
(see Section \ref{sect:examples}, 
Section \ref{sect:AB} and Section \ref{sect:DA} for the details). 
Firstly, we construct a generalization of the Gold functions which is the most typical APN functions \cite{Gold1968maximal, Nyberg1994differentially}: 
\begin{theorem}\label{thm:main0}
A monomial function $f : F \rightarrow F$ defined by 
\begin{align*}
f (x) = x^{p^i + p-1} \ \ (i > 0 \ \mbox{and} \ \gcd (i, n) = 1). 
\end{align*}
is a GAPN function of algebraic degree $p$. 
\end{theorem}
Secondly, when $p = 3$, we obtain a partial generalization of a relation between APN functions and AB functions introduced in \cite{CV1995links}: 
\begin{theorem}\label{thm:main1}
Suppose that $ p=3 $. 
Let $ f $ be a function of algebraic degree at most $3$ with the condition 
$ f(-x)=-f(x) $ for any $ x \in \mathbb{F}_{3^n} $. 
Then if $f$ is a generalized almost bent function, then $ f $ is a GAPN function. 
Here generalized almost bent functions are defined in Section \ref{sect:AB}. 
\end{theorem}
Thirdly, we obtain a generalization of the construction of dual arcs associated with 
APN functions introduced in \cite{Yoshiara2008dimensional}: 
\begin{theorem}\label{thm:main2}
We can construct dual arcs with GAPN functions of algebraic degree $ p $. 
\end{theorem}

This paper is organized as follows. 
In Section \ref{sect:criterions}, 
we give several characterizations for GAPN functions, which are generalizations of classical results for APN functions on $ \mathbb{F}_{2^{n}} $.  
In Section \ref{sect:examples}, 
we raise two examples of GAPN functions. 
One is the inverse permutation and the other is a generalization of the Gold functions. 
In Section \ref{sect:AB}, 
we define a generalization of almost bent functions and prove the Theorem \ref{thm:main1}.  
In Section \ref{sect:DA}, 
we construct dual arcs with GAPN functions of algebraic degree $ p $. 

\section{Characterizations of GAPN functions}\label{sect:criterions}

\subsection{The property of stability of GAPN functions}

Two functions $f$ and $g$ are called \textbf{extended affine equivalent}
(EA-equivalent) if $g = A_1 \circ f \circ A_2 + A_0$, where $ A_{1} $ and $ A_{2} $ are affine permutation and $ A_0 $ is an affine function. 
We see that EA-equivalence preserves the set 
\begin{align*}
{\cal N}_f := \Set{ \tilde{N}_{f} (a, b) | a \in F^{\times}, b \in F }. 
\end{align*}
The following proposition is a generalization of \cite[Proposition 1]{Nyberg1994differentially}. 

\begin{proposition}
Let $f$, $g \colon F \rightarrow F$ be EA-equivalent functions. 
Then ${\cal N}_f = {\cal N}_g$. 
In particular, $f$ is a GAPN function if and only if $g$ is a GAPN function. 
\end{proposition}

\begin{proof}
By definition, we have 
$g = A_1 \circ f \circ A_2 + A_0$ for some affine permutations $ A_{1},A_{2} $ and affine function $ A_{0} $. 
For each $ i \in \Set{0,1,2} $, we may put $A_i = \alpha_i + c_i$, 
where $\alpha_i$ is a linear function on $F$ and $c_i \in F$. 
Then $\alpha_1$ and $\alpha_2$ are bijective. 
We have 
\begin{align*}
\sum_{i \in \mathbb{F}_p} A_0 (x + i a) 
= \sum_{i \in \mathbb{F}_p} \left( \alpha_0 (x + i a) + c_0 \right) 
= \alpha_{0}(a)\sum_{i \in \mathbb{F}_{p}}i = \alpha_{0}(a)r  
\end{align*}
for any $ a \in F^{\times} $, where $ r $ denotes $ \sum_{i \in \mathbb{F}_{p}}i $. 
Then we obtain 
\begin{align*}
\tilde{D}_a g (x) 
&= \sum_{i \in \mathbb{F}_p} \left( A_1 \circ f \circ A_2 + A_0 \right) (x + i a) 
= \sum_{i \in \mathbb{F}_p} \left( 
\alpha_1 \left( f (\alpha_2 (x + i a) + c_2 ) \right) + c_1 \right) + \alpha_{0}(a)r \\
&= \alpha_1 \left( \sum_{i \in \mathbb{F}_p} 
f \left( A_{2}(x) + i \alpha_2 (a) \right) \right) + \alpha_{0}(a)r
= \alpha_1 \left( \tilde{D}_{\alpha_2 (a)} f (A_2(x)) \right) + \alpha_{0}(a)r. 
\end{align*}
Hence for any $ a \in F^{\times} $ and $b \in F$, $\tilde{D}_a g (x) = b$ if and only if 
$\tilde{D}_{\alpha_2 (a)} f (A_2 (x)) = \alpha_1^{-1} (b-\alpha_{0} (a)r)$. 
Since $A_2$ is a permutation, we obtain 
$\tilde{N}_g (a, b) = \tilde{N}_f (\alpha_2 (a), \alpha_1^{-1} (b-\alpha_{0}(a)r))$ for any $a \in F^{\times}$ and 
$b \in F$. Thus ${\cal N}_f = {\cal N}_g$. 
\end{proof}

\begin{remark}
In \cite{CCZ1998codes}, 
Carlet, Charpin and Zinoviev showed that
EA-equivalence is a particular case of CCZ-equivalence 
and every permutation is CCZ-equivalent to its inverse. 
Here CCZ-equivalence corresponds to the affine equivalence of the graphs of functions, 
that is, functions $f$ and $g$ are CCZ-equivalent if and only if, for some affine permutation, 
the image of the graph of $f$ is the graph of $g$. 
When $p=2$, for any CCZ-equivalent functions $f$ and $g$, 
we have ${\cal N}_f = {\cal N}_g$. In particular, $f$ is APN if and only if $g$ is APN 
\cite{BCP2006new}. 
Unfortunately, this property is not extended for GAPN functions. 
For example, the function $f \colon \mathbb{F}_{3^5} \to \mathbb{F}_{3^5}$ defined by 
$f(x) = x^{57}$, which is the composition of $f_1(x) = x^{19}$ and the Frobenius mapping 
${\rm Fb} (x) = x^3$, is a GAPN function, since ${\rm Fb}$ is linear 
and $f_1$ is a GAPN function by 
Proposition \ref{GGF}. However we can check easily that the inverse function 
$f^{-1} (x) = x^{17}$ is not a GAPN function (see Remark \ref{bad points}, (1) for the detail). 
\end{remark}

\subsection{GAPN functions of algebraic degree $ p $}
For a function $ f \colon F \rightarrow F $ and a positive integer $ m $, we define a function 
\begin{gather*}
[f]^{m} \colon F^{m} \longrightarrow F, \ \ 
[f]^{m}(x_{1}, \dots, x_{m}) \coloneqq \sum_{I \subset [m]}(-1)^{m-|I|}f\left(\sum_{i \in I}x_{i}\right), 
\end{gather*}
where $ [m] $ denotes the set $ \{1, \dots, m\} $. 
We also define $ [f]^{0} \coloneqq f(0) $. 
For example 
\begin{gather*}
[f]^{1}(x) = f(x)-f(0), \ \ [f]^{2}(x,y) = f(x+y)-f(x)-f(y)+f(0), \\
[f]^{3}(x,y,z) = f(x+y+z)-f(x+y)-f(x+z)-f(y+z) \\ +f(x)+f(y)+f(z)-f(0). 
\end{gather*}
It is easy to verify the following: 
\begin{proposition}\label{prop:mult_rec}
Let $ m $ be a positive integer. 
Then
\begin{align*}
& \quad [f]^{m+1}(x,y,z_{1}, \dots, z_{m-1}) \\
&= [f]^{m}(x+y,z_{1}, \dots, z_{m-1}) - [f]^{m}(x,z_{1}, \dots, z_{m-1}) - [f]^{m}(y,z_{1},\dots, z_{m-1})
\end{align*}
for any $ x,y,z_{1}, \dots, z_{m-1} \in F $. 
\end{proposition}

Every function $ f \colon F \rightarrow F $ can be represented uniquely by a polynomial 
$\displaystyle 
f(x) = \sum_{d = 0}^{p^n-1} c_d x^d \in F [x]$. Each $d$ has the $p$-adic expansion 
$\displaystyle 
d = \sum_{s = 0}^{n-1} d_s p^s
$, where $0 \leq d_s < p$. 
Let $w_p (d)$ denote the integer $\displaystyle 
d = \sum_{s = 0}^{n-1} d_s$, and we call it the \textbf{$p$-weight} of $d$. 
Then we can write $d=p^{s_1} + \cdots + p^{s_w}$, where $w = w_p (d)$ and it is not necessary that $s_i$'s are distinct. 

\begin{lemma}\label{lem:order_algdeg_mono}
For any integer $m$ such that $m \geq w$, we have 
\begin{align} \label{want lem:order_algdeg_mono}
\left[ x^d \right]^m = \left\{ 
\begin{array}{cl}
0 & (m \geq w + 1), \\
\displaystyle \sum_{\sigma \in {\mathfrak S}_m} 
x_1^{p^{s_{\sigma(1)}}} \cdots x_m^{p^{s_{\sigma(m)}}}
& (m = w), \\
\end{array}
\right.
\end{align}
where ${\mathfrak S}_m$ is the symmetric group of degree $m$. 
In particular, we have 
$[f]^m = 0$ if $m > \max \Set{ w_p (d) | c_d \ne 0 }$. 
\end{lemma}

\begin{proof}
For any subset $I \subset [m]$, we have 
\begin{align*}
\left( \sum_{i \in I} x_i \right)^d = \left( \sum_{i \in I} x_i \right)^{p^{s_1} + \cdots + p^{s_w}}
= \prod_{j=1}^{w} \left( \sum_{i \in I} x_i ^{p^{s_j}} \right) 
= \sum_{t_1, \dots, t_w \in I} x_{t_1}^{p^{s_1}} \cdots x_{t_w}^{p^{s_w}}. 
\end{align*}
Hence we obtain 
\begin{align*}
\left[ x^d \right]^m &= \sum_{I \subset [m]}(-1)^{m-|I|} \left(\sum_{i \in I}x_{i}\right) ^d 
= \sum_{I \subset [m]}(-1)^{m-|I|} 
\left( \sum_{t_1, \dots, t_w \in I} x_{t_1}^{p^{s_1}} \cdots x_{t_w}^{p^{s_w}} \right) \\
&= \sum_{t_1, \dots, t_w \in [m]} 
\left( \sum_{\{ t_1, \dots, t_w \} \subset I \subset [m]} (-1)^{m - |I|} \right) 
x_{t_1}^{p^{s_1}} \cdots x_{t_w}^{p^{s_w}}. 
\end{align*}
Let $T \coloneqq \# \left\{ t_1 , \dots , t_w \right\} \leq w \leq m$. 
Then we have 
\begin{align*}
\# 
\Set{ I \subset [m] | \{ t_1, \dots, t_w \} \subset I \ \ \mbox{and} \ \ |I| = T + j } 
= \binom{m-T}{j} \ \ \ (0 \leq j \leq m - T), 
\end{align*}
where $\binom{m-T}{j}$ denotes the binomial coefficients. 
Thus we obtain 
\begin{align*}
\sum_{\{ t_1, \dots, t_w \} \subset I \subset [m]} (-1)^{m - |I|} 
= \sum_{j = 0}^{m -T} (-1)^{(m-T)-j} \binom{m-T}{j}
= \left\{ 
\begin{array}{cc}
0 & (T < m), \\
1 & (T = m). \\
\end{array}
\right.
\end{align*}
Therefore we obtain (\ref{want lem:order_algdeg_mono}). 
\end{proof}

\begin{definition}
Let $ f \colon F \rightarrow F $ be a non-zero function. 
The maximum integer $ m $ such that $ [f]^{m} \neq 0 $ is called the \textbf{algebraic degree} of $ f $, denoted by $ d^{\circ}(f) $.  
A function of algebraic degree $ 2 $ is called \textbf{quadratic}. 
\end{definition}

\begin{proposition}\label{prop:ch_algdeg}
Let $ f \colon F \rightarrow F $ be a function and let $m$ be a positive integer. 
\begin{enumerate}[label=(\arabic*)]
\item $ d^{\circ}(f)=0 $ if and only if $ f $ is a non-zero constant function. 
\item $ d^{\circ}(f) = m $ if and only if $ [f]^{m} $ is a non-zero 
$ \mathbb{F}_{p} $-multilinear form. 
\item Let $\displaystyle 
f(x) = \sum_{d = 0}^{p^n-1} c_d x^d \in F [x]$. Then 
$d^{\circ} (f) \leq \max \Set{ w_p (d) | c_d \ne 0 }$. 
\end{enumerate}

\end{proposition}
\begin{proof}
Clear from Proposition \ref{prop:mult_rec} and Lemma \ref{lem:order_algdeg_mono}. 
\end{proof}

EA-equivalent preserves algebraic degrees of functions, that is, we have 

\begin{proposition}
Let $f$, $g \colon F \rightarrow F$ be EA-equivalent functions, and let 
$d^{\circ} (f) \geq 2$. 
Then $d^{\circ} (g) = d^{\circ} (f)$. 
\end{proposition}

\begin{proof}
By definition, we have 
$g = A_1 \circ f \circ A_2 + A_0$ for some affine functions 
$A_0$, $A_1$ and $A_2$, where $A_1$ and $A_2$ are permutations. 
For each $i \in \Set{ 0, 1, 2}$, we may put $A_i = \alpha_i + c_i$, 
where $\alpha_i$ is a linear function on $F$ and $c_i \in F$. 
Then $\alpha_1$ and $\alpha_2$ are bijective. 
For any integer $m \geq 2$, 
we have 
\begin{align*}
[A_0]^m (x_1, \dots , x_m) 
= \alpha_0 \left( \sum_{I \subset [m]} (-1) ^{m-|I|} \sum_{i \in I} x_i \right) 
+ c_0 \sum_{I \subset [m]} (-1) ^{m-|I|} = 0. 
\end{align*}
Hence we obtain 
\begin{align*}
[g]^m (x_1, \dots, x_m) &= [A_1 \circ f \circ A_2 + A_0]^m (x_1, \dots, x_m) 
= [A_1 \circ f \circ A_2]^m (x_1, \dots, x_m) \\
&= \alpha_1 \left(  
[f]^{m+1} (\alpha_2 (x_1), \dots, \alpha_2 (x_m), c_2) + [f]^m (\alpha_2 (x_1), \dots, \alpha_2 (x_m))
\right). 
\end{align*}
Therefore we have $d^{\circ} (g) = d^{\circ} (f)$. 
\end{proof}

For a function $ f \colon F \rightarrow F $ we define 
$\tilde{B}_{f}(x,y) \coloneqq [f]^{p}(x,y, \dots, y) $. 
Note that 
if $ d^{\circ}(f) \leq p $ then $ \tilde{B}_{f}(x,y) $ is linear in $ x $ by Proposition \ref{prop:ch_algdeg} 
and when $ p=2 $ a function $ f $ is quadratic if and only if 
$ \tilde{B}_{f}(x,y)=f(x+y)+f(x)+f(y)+f(0) $ is a non-zero bilinear form. 

\begin{proposition}\label{prop:Bf}
$\tilde{B}_{f}(x,a) = \tilde{D}_{a}f(x)-\tilde{D}_{a}f(0)$ 
for any $x$, $a \in F$. 
\end{proposition}

\begin{proof}
Since 
$\displaystyle 
\tilde{B}_{f}(x,a) = [f]^{p}(x, a, \dots, a)
=\sum_{i=0}^{p-1}(-1)^{p-1-i}\binom{p-1}{i}\left(f(x+ia)-f(ia)\right)
$, the result holds true by the congruence 
$\displaystyle 
\binom{p-1}{i} \equiv (-1)^{i}\pmod{p}$. 
\end{proof}

\begin{proposition} \label{prop:decomp}
Suppose that $ d^{\circ}(f) \leq p $.
Then 
\begin{align*}
\tilde{D}_a f (x \pm y) 
= \tilde{D}_a f (x) \pm \tilde{D}_a f (y) 
\mp \tilde{D}_a f (0) . 
\end{align*}
In particular, if $\tilde{D}_a f (0) = 0$, then the mapping $\tilde{D}_a f$ is linear over ${\mathbb F}_{p}$. 
\end{proposition}

\begin{proof}
By Proposition \ref{prop:Bf}, we have 
\begin{align*}
\tilde{D}_{a}f(x \pm y) 
&= \tilde{B}_{f}(x \pm y,a)+\tilde{D}_{a}f(0) 
= \tilde{B}_{f}(x,a) \pm \tilde{B}_{f}(y,a) + \tilde{D}_{a}f(0) \\
&= \left(\tilde{D}_{a}f(x) - \tilde{D}_{a}f(0)\right)
\pm \left(\tilde{D}_{a}f(y) - \tilde{D}_{a}f(0)\right)
+ \tilde{D}_{a}f(0) \\
&= \tilde{D}_a f (x) \pm \tilde{D}_a f (y) 
\mp \tilde{D}_a f (0). 
\end{align*}
\end{proof}

We have two characterizations as follows 
for GAPN functions of algebraic degree at most $ p $. 
These are generalizations of classical results for quadratic APN functions. 

\begin{proposition} \label{lem:alg_criterion1}
Suppose that $ d^{\circ}(f) \leq p $. 
Then $ \tilde{N}_{f}(a,b) $ equals zero or $ \tilde{N}_{f}(a,\tilde{D}_{a}f(0)) $ for any $ a \in F^{\times} $ and $ b \in F $. 
In particular, $f$ is a GAPN function if and only if $ \tilde{N}_{f}(a,\tilde{D}_{a}f(0)) \leq p $ for any $ a \in F^{\times} $. 
\end{proposition}
\begin{proof}
If $\tilde{D}_a f (x) = b$ has no solutions in $F$, 
then $\tilde{N}_f (a, b) = 0$. Assume that $x_0 \in F$ is a solution of 
$\tilde{D}_a f (x) = b$. 
By Proposition \ref{prop:decomp}
\begin{align*}
\tilde{D}_{a}f(x) - b 
= \tilde{D}_{a}f(x) - \tilde{D}_{a}f(x_{0}) 
= \tilde{D}_{a}f(x-x_{0})-\tilde{D}_{a}f(0). 
\end{align*}
Hence $\tilde{D}_{a}f(x)=b $ if and only if $\tilde{D}_{a}f(x-x_{0}) = \tilde{D}_{a}f(0)$, 
and hence we have that $ \tilde{N}_{f}(a,b) = \tilde{N}_{f}(a,\tilde{D}_{a}(0)) $. 
\end{proof}

Since $\tilde{B}_f (x, a) =0$, that is $\tilde{D}_{a}f(x) = \tilde{D}_{a}f(0)$ has 
trivial solutions $ x \in \mathbb{F}_{p}a $, and $[f]^p = 0$ implies that 
$\tilde{B}_f (x, a) = [f]^p (x, a, \dots, a) = 0$ for any $x$, $a \in F$, 
we obtain 

\begin{proposition}\label{lem:alg_criterion2}
\begin{enumerate}[label=(\arabic*)]
\item Suppose that $ d^{\circ}(f) \leq p $. 
Then $ f $ is a GAPN function if and only if 
$ \Set{ x \in F | \tilde{B}_f (x, a) =0 } = {\mathbb F}_p a $ for any $ a \in F^{\times} $. 
\item 
If $f$ is a GAPN function with $d^{\circ} (f) \leq p$, then $d^{\circ} (f) = p$. 
In particular, 
GAPN functions are algebraic degree at least $p$. 
\end{enumerate}
\end{proposition}

\subsection{Fourier-Walsh transform}
For a function $f \colon F \rightarrow F$ and an element $ b \in F $, we define 
\begin{align*}
f_{b} \colon F \longrightarrow \mathbb{F}_{p}, \ x \longmapsto \Tr (b f (x)), 
\end{align*}
where $ \Tr $ denotes the absolute trace on $ F $. 
The functions $f_{b}$ are called the \textbf{components} of $f$. 
For any function $f \colon F \rightarrow \mathbb{F}_{p}$, 
let $\mathcal{F}(f)$ denote 
the following value related to the Fourier-Walsh transform of $f$: 
\begin{align*}
\mathcal{F}(f) \coloneqq \sum_{x \in F} \zeta_p ^{f (x)}, 
\end{align*}
where $\zeta_p$ is the primitive $p$-th root of unity. 
We have the following characterization for GAPN functions, which is a generalization of 
APN's one 
introduced in \cite{Nyberg1995s-boxes}.

\begin{proposition} \label{char_of_GAPN}
Let $f \colon F \rightarrow F$ be a function.
Then 
\begin{align*}
\sum_{a \in F, 
b \in F^{\times}} 
|\mathcal{F} (\tilde{D}_a f _b)|^2 \geq p^{2 n + 1} (p^n - 1)
\end{align*}
with equality if and only if $f$ is a GAPN function.  
\end{proposition}

\begin{proof}
We define $p^n \times p^n$ matrices $ X,T,N $ which are indexed by elements in $ F \times F $. 
The $(a, b)$-components of these matrices are as follows: 
\begin{gather*}
X_{a b} \coloneqq \zeta_p^{\Tr (a b)}, \ 
T_{a b} \coloneqq \mathcal{F} (\tilde{D}_a f_b), \ 
N_{a b} \coloneqq \tilde{N}_f (a, b). 
\end{gather*}
Then we have $ T=NX $ since
\begin{align*}
T_{a b} = \sum_{x \in F} \zeta_p^{\Tr (b \tilde{D}_a f (x))} 
= \sum_{y \in F} \tilde{N}_{f} (a, y) \zeta_p ^{\Tr (y b)} 
= \sum_{y \in F} N_{a y} X_{y b}. 
\end{align*}
Moreover, we have $ XX^{\ast} = p^{n}I $, where $ X^{\ast} $ denotes the adjoint matrix of $ X $ and $ I $ the identity matrix, since 
\begin{align*}
\sum_{c \in F} X_{a c} \overline{X_{c b}} 
= \sum_{c \in F} \zeta_p^{{\rm Tr} \left( (a-b) c \right)}
= \left\{ \begin{array}{cc}
p^n & (a = b), \\
0 & (a \neq b). 
\end{array}
\right. 
\end{align*}
Therefore we have 
\begin{align*}
\sum_{a, b \in F} |T_{a b}|^2 &= \Tr \left( T T^{\ast} \right)  
= \Tr \left( N X X^{\ast} N^{\ast} \right)  
= p^n \Tr \left( N N^{\ast} \right)
= p^n \sum_{a, b \in F} \tilde{N}_f (a, b)^2. 
\end{align*}
On the other hand, 
we have 
$\displaystyle 
\tilde{N}_f (0, b)^2 = \left\{ \begin{array}{cc}
p^{2 n} & (b = 0), \\
0 & (b \not = 0), 
\end{array}
\right.
$
and if $a \not = 0$, then we have $\tilde{N}_f (a, b)^2 \geq p \tilde{N}_f (a, b)$. 
Hence we obtain 
\begin{align*}
\sum_{a, b \in F} |T_{a b}|^2 
&= p^n \left( 
\sum_{b \in F} \tilde{N}_f (0, b)^2 
+ \sum_{a \in F^{\times}, b \in F} \tilde{N}_f (a, b)^2 \right) \\
&= p^{3 n} + p^n \sum_{a \in F^{\times}, b \in F} \tilde{N}_f (a, b)^2 
\geq p^{3n} + p^{n+1} \sum_{a \in F^{\times}, b \in F} \tilde{N}_f (a, b). 
\end{align*}
Moreover, we have 
$\displaystyle 
\sum_{b \in F} \tilde{N}_f (a, b) 
= \sum_{b \in F} \# \left( (\tilde{D}_a f )^{-1} (b) \right) = p^n$, and hence we obtain 
$\displaystyle 
\sum_{a \in F^{\times}, b \in F} \tilde{N}_f (a, b) = (p^n - 1)p^n$. 
We have 
$\displaystyle \sum_{a \in {\mathbb F}_{p^n}} T _{a 0}^2 
=\sum_{a \in {\mathbb F}_{p^n}} \left( \sum_{x \in {\mathbb F}_{p^n}} \zeta_p^{{\rm Tr} (0 \cdot D_a f (x))} \right)^2 = p^{3 n}$ clearly. 
Thus we have 
\begin{align*}
\sum_{a \in F, 
b \in F^{\times}} |\mathcal{F} (\tilde{D}_a f _b)|^2 
= \sum_{a \in F, b \in F^{\times}} |T_{a b}|^2 
\geq p^{2n+1} (p^{n} - 1) 
\end{align*}
with equality if and only if $\tilde{N}_f (a, b)$ equals $0$ or $p$ for all $ a \in F^{\times} $ and $b \in F$, that is, $f$ is a GAPN function. 
\end{proof}

\section{Examples of GAPN functions}\label{sect:examples}

\subsection{Inverse permutations}
The inverse permutation $ f $ on $ F $ is defined by 
\begin{align*}
f(x) \coloneqq x^{p^{n}-2} 
=\begin{cases}
x^{-1} & (x \neq 0), \\
0 & (x=0).
\end{cases}
\end{align*}
The following is well known: 
\begin{proposition}[Beth-Ding \cite{BD1994almost}, 
Nyberg \cite{Nyberg1994differentially}]
Let $ f $ be the inverse permutation on $ \mathbb{F}_{2^{n}} $. 
Then $ f $ is APN if and only if $ n $ is odd. 
\end{proposition}
This proposition is generalized as follows: 
\begin{proposition}
Let $ p $ be an odd prime. 
Then the inverse permutation on $ F $ is a GAPN function. 
\end{proposition}
\begin{proof}
For convenience let $ 0^{-1} \coloneqq 0 $. 
We consider an equation 
\begin{align*}
\sum_{i \in \mathbb{F}_{p}}(x+ia)^{-1} = b, 
\end{align*}
where $ a \in F^{\times} $ and $ b \in F $. 
First suppose that there exists a solution $ x \not \in \mathbb{F}_{p}a $. 
Multiplying the equation by $ \prod_{i \in \mathbb{F}_{p}}(x+ia) $ we have  
\begin{align*}
b\prod_{i \in \mathbb{F}_{p}}(x+ia) + \text{( a polynomial in $ x $ with degree at most $ p-1 $)} = 0. 
\end{align*}
Since every element in $ x + \mathbb{F}_{p}a $ is a solution, we have $ b \neq 0 $ and the number of solutions outside $ \mathbb{F}_{p}a $ is exactly $ p $. 

Next we suppose that $ x \in \mathbb{F}_{p}a $ is a solution. 
Then we have 
\begin{align*}
b = \sum_{i \in \mathbb{F}_{p}}(x+ia)^{-1} = \sum_{i \in \mathbb{F}_{p}}(ia)^{-1} = a^{-1}\sum_{i=1}^{p-1} i^{-1} = a^{-1}\sum_{i = 1}^{p-1}i = 0. 
\end{align*}
Hence it is impossible that the equation has a solution in $ \mathbb{F}_{p}a $ and a solution outside $ \mathbb{F}_{p}a $ simultaneously. 
Therefore $ \tilde{N}_{f}(a,b) \leq p $ for any $ a \in F^{\times} $ and $ b \in F $, that is the inverse permutation is a GAPN function. 
\end{proof}

\subsection{Generalized Gold functions}
When $ p=2 $ the most typical quadratic APN functions are the Gold functions 
\cite{Gold1968maximal, Nyberg1994differentially}, which are defined by 
\begin{align*}
f(x) = x^{2^{i}+1} \text{ with } \gcd(n,i)=1. 
\end{align*}
In this subsection, we construct a generalization of the Gold function. 

\begin{proposition} \label{GGF}
Let $f$ be a monomial function defined by 
\begin{align*}
f (x) = x^{1 +  p^{i_2} + \cdots + p^{i_{p}}} \ \ 
(i_2, \dots, i_{p} \geq 0, \ (i_2, \dots, i_{p}) \not = (0, \cdots, 0) ). 
\end{align*}
Then 
\begin{enumerate}[label=(\roman*)]
\item $ d^{\circ}(f) \leq p $. 
\item
$\tilde{B}_f (x, a) 
= (p-1) 
\left( a^{d - 1} x + a^{d - p^{i_2} }x^{p^{i_2}} + \cdots + a ^{d - p^{i_{p}}}x^{p^{i_{p}}} \right)$ 
for any $a \in F^{\times}$, where 
$d = 1 +  p^{i_2} + \cdots + p^{i_{p}}$. 
\item Assume that $ \Set{ x \in F | x + x^{p^{i_2}} + \cdots + x^{p^{i_{p}}} = 0 } 
= \mathbb{F}_p $. 
Then $f$ is a GAPN function of algebraic degree $p$. 
\end{enumerate}
\end{proposition}

\begin{proof}
Since the $p$-weight of $f$ is $w_p (d) = p$, 
the statement (i) is clear from Lemma \ref{lem:order_algdeg_mono}. 
We prove the statement (ii). 
When $p=2$, 
we have 
\begin{align*}
\tilde{B}_f (x, a) 
&= f (x+a) + f (x) + f (a) + f(0) 
= (x+a) (x^{2^{i_2}} + a^{2^{i_2}}) + x^{1 + 2^{i_2}} + a^{1 + 2^{i_2}} \\
&= a x^{2^{i_2}} + a^{2^{i_2}} x. 
\end{align*}
When $p \geq 3$, let $i_1 = 0$. Then we have 
\begin{align*}
\tilde{D}_a f (0)
= \left( \sum_{j \in {\mathbb F}_p} j^{p^{i_1} + \cdots +p^{i_p}} \right)
a^{p^{i_1} + \cdots +p^{i_p}}
= \left( \sum_{j \in {\mathbb F}_p} j \right)
a^{p^{i_1} + \cdots +p^{i_p}} = 0. 
\end{align*}
Hence we obtain 
\begin{align*}
\tilde{B}_{f} (x, a) &= \tilde{D}_a f (x) - \tilde{D}_a f (0) 
= \tilde{D}_a f (x) 
= \sum_{j \in {\mathbb F}_p} 
\left( \prod_{\ell = 1}^{p} \left( x^{p^{i_\ell}} + (j a)^{p^{i_{\ell}}} \right) \right) \\
&= \sum_{j \in {\mathbb F}_p} 
\left( \prod_{\ell = 1}^{p} \left( x^{p^{i_\ell}} + j a^{p^{i_{\ell}}} \right) \right) 
= \sum_{j \in {\mathbb F}_p} 
\left( 
\sum_{K \subset [p]} 
j^{|K|} a^{\sum_{k \in K} p^{i_k}} x^{\sum_{k \in [p] \setminus K} p^{i_k}}
\right) \\
&= \sum_{K \subset [p]} 
\left( \sum_{j \in {\mathbb F}_p} j^{|K|} \right)
a^{\sum_{k \in K} p^{i_k}} x^{\sum_{k \in [p] \setminus K} p^{i_k}}. 
\end{align*}
Since we have 
$\displaystyle 
\sum_{j \in {\mathbb F}_p} j^{|K|} = \left\{ 
\begin{array}{cl}
0 & (|K| \not = p - 1), \\
p-1 & (|K| = p - 1),  
\end{array}
\right.
$ we obtain the desired equation. 

We prove the statement (iii). 
Since $a \not = 0$, by the assumption and (ii), we have 
\begin{align*}
\Set{ x \in F | \tilde{B}_f(x, a) = 0 }
= \Set{ a y |  
y + y^{p^{i_2}} + \cdots + y^{p^{i_{p}}} = 0 } 
= {\mathbb F}_p a. 
\end{align*}
Hence $f$ is a GAPN function with $d^{\circ} (f) = p$ by Proposition \ref{lem:alg_criterion2}. 
\end{proof}

By Proposition \ref{GGF}, we obtain a generalization of the Gold function: 

\begin{corollary} \label{GGF2}
Let $f : F \rightarrow F$ be a monomial function defined by 
\begin{align*}
f (x) = x^{p^i + p-1} \ \ (i > 0 \ \mbox{and} \ \gcd (i, n) = 1). 
\end{align*}
Then $f$ is a GAPN function of algebraic degree $p$. 
We call them the \textbf{generalized Gold functions}. 
\end{corollary}

\begin{proof}
In Proposition \ref{GGF}, let $(i_2, i_3,  \dots , i_{p}) = (i, 0, \cdots , 0)$ with $i > 0$. 
Then by (iii) in Proposition \ref{GGF}, 
the monomial function 
$
f (x) = x^{p^i + p-1}
$ 
is a GAPN function of algebraic degree $p$, if we have 
\begin{align} \label{roots of unity}
\left\{ x \in F \; \middle| \;  x^{p^{i}} = x \right\} 
= {\mathbb F}_p, \ \mbox{that is}, \  
\left\{ x \in F \; \middle| \;  x^{p^{i} - 1} = 1 \right\} 
= {\mathbb F}_p^{\times}. 
\end{align}
On the other hand, since $\gcd (i, n) = 1$, we have 
$\gcd (p^i - 1, p^n -1) = p-1$. 
In fact, 
we have 
$\displaystyle 
\gcd (p^i - 1, p^n -1) \mid \left( (p^n-1) - (p^i-1) \right) = \pm p^{\min \{i, n \}} 
\left( p^{|n-i|} - 1 \right), 
$
and hence we have $\gcd (p^i - 1, p^n -1) \mid \left( p^{|n-i|} - 1 \right)$. 
Since $\gcd (i, n) = 1$, by induction on $\max \{i, n \}$, we obtain 
$\gcd (p^i - 1, p^n -1) = p-1$. 
Therefore we have 
\begin{align*}
\# \left\{ x \in F \; \middle| \;  x^{p^{i} - 1} = 1 \right\} 
= \gcd (p^i -1, p^n-1) = p-1, 
\mbox{and hence, we obtain (\ref{roots of unity})}. 
\end{align*}
\end{proof}

When $p=2$, there are no quadratic APN functions on $\mathbb{F}_{2^n}$ of the form 
\begin{align*}
f (x) = \sum_{i = 1}^{n-1} c_i x^{2^i + 1}, \ \ \ c_i \in \mathbb{F}_{2^n}
\end{align*}
except the Gold functions \cite{BCCL2006almost}. 
Unfortunately, this property is not generalized for GAPN functions. 
In fact, we have 
\begin{proposition}
Assume that $p$ is an odd prime and $n$ is odd. 
Then the function $f \colon F \longrightarrow F$ defined by 
\begin{align*}
f(x) = x^{p^i + p-1} - x^{p^{n-i} + p-1} \ \ \ \left( i>0 \ \ \mbox{and} \ \ \gcd(i,n)=1 \right)
\end{align*}
is a GAPN function of algebraic degree $p$. 
\end{proposition}
\begin{proof}
Clearly, $d^{\circ} (f) \leq p$, and $\tilde{D}_a f (0) = 0$ for any $a \in F^{\times}$. 
Thus all we have to do is to show that 
$\tilde{N}_f (a, 0) \leq p$ for any $a \in F^{\times}$. 
Then we have 
\begin{align*}
\tilde{D}_a f (x) 
&= \left( a^{p^{i} + p-2} x - a^{p-1} x^{p^i} \right) 
- \left( a^{p^{n-i} + p-2} x - a^{p-1} x^{p^{n-i}} \right) \\
&= a^{p-1} x \left( - x^{p^i - 1} + x^{p^{n-i}-1} + a^{p^i - 1} - a^{p^{n-i}-1} \right). 
\end{align*}
Hence it is sufficient to show that the equation $- x^{p^i - 1} + x^{p^{n-i}-1} + a^{p^i - 1} - a^{p^{n-i}-1}$ has only trivial $p-1$ solutions $a$, $2a$, $\dots$, $(p-1) a$ for any $a \in F^{\times}$. 
It follows immediately from Lemma \ref{(p-1)-to-1}. 
\end{proof}
\begin{lemma} \label{(p-1)-to-1}
The mapping $\varphi \colon F^{\times} \to F$ defined by 
$
\varphi(a) = a^{p^i - 1} - a^{p^{n-i}-1}
$ 
is $(p-1)$-to-$1$. 
\end{lemma}
\begin{proof}
We consider the composition of $\varphi$ and the Frobenius automorphism 
${\rm Fb} (x) = x^{p^i}$. Then we have 
\begin{align*}
{\rm Fb} \circ \varphi (a) = \left( a^{p^i - 1} - a^{p^{n-i}-1} \right)^{p^i} 
= \left( a^{p^i - 1} \right)^{p^i} - \frac{1}{a^{p^{i}-1}} 
= \psi_2 \circ \psi_1 (a), 
\end{align*}
where $\psi_1$ and $\psi_2$ are defined by 
\begin{align*}
\psi_1 \colon F^{\times} \longrightarrow F^{\times}, \ \ 
a \longmapsto a^{p^i-1}, \ \ \mbox{and} \ \ 
\psi_2 \colon F^{\times} \longrightarrow F, \ \ 
\alpha \longmapsto \alpha^{p^i} - \frac{1}{\alpha}. 
\end{align*}
Since $\rm Fb$ is a bijection, it is sufficient to show the following two properties: 
\begin{itemize}
\item $\psi_1 \colon F^{\times} \rightarrow F^{\times}$ is a $(p-1)$-to-$1$ mapping. 
\item $\psi_2$ is injective on ${\rm Im} (\psi_1)$. 
\end{itemize}
We show the first property. 
For any two elements $a$ and $b \in F^{\times}$ such that $a^{p^i-1} = b^{p^i-1}$, 
we have $\left( a/b \right)^{p^i - 1} = 1$. 
Since $\gcd (i, n) = 1$, we obtain that 
$a/b$ is contained in $\mathbb{F}_p^{\times}$. 
Hence $\psi_1$ is a $(p-1)$-to-$1$ mapping. 
Next we show the second property. Since ${\rm Im} (\psi_1)$ is the subgroup of 
$F^{\times}$ whose cardinality equals $\frac{p^n-1}{p-1}$, we obtain 
${\rm Im} (\psi_1) = \langle \gamma^{p-1} \rangle$, 
where $\gamma$ is a generator of $F^{\times}$. 
Let $\gamma^{(p-1) m_1}$ and $\gamma^{(p-1) m_2}$ be two elements in ${\rm Im} (\psi_1)$ such that 
\begin{gather*}
\left( \gamma^{(p-1)m_1} \right)^{p^i} - \frac{1}{\gamma^{(p-1)m_1}}
= \left( \gamma^{(p-1)m_2} \right)^{p^i} - \frac{1}{\gamma^{(p-1)m_2}}, \\ 
\mbox{that is}, \ \ 
\gamma^{(p-1)(m_1 + m_2)} \left( \gamma^{(p-1)m_1} - \gamma^{(p-1)m_2} \right)^{p^i} 
= - \left( \gamma^{(p-1)m_1} - \gamma^{(p-1)m_2} \right). 
\end{gather*}
Assume that $\gamma^{(p-1)m_1} \neq \gamma^{(p-1)m_2}$. Then 
$\frac{p^n-1}{p-1} = 1 + p + \cdots + p^{n-1}$ is odd, since $n$ is odd. 
Hence we have 
\begin{align*}
\left( \left( \gamma^{(p-1)m_1} - \gamma^{(p-1)m_2} \right)^{\frac{p^n-1}{p-1}} \right)^{p^i-1} 
&= 
\left( 
\gamma^{(p-1)(m_1+m_2)} \left( \gamma^{(p-1)m_1} - \gamma^{(p-1)m_2} \right)^{p^i-1} 
\right)^{\frac{p^n-1}{p-1}} \\
&= \left( -1 \right)^{\frac{p^n-1}{p-1}} = - 1. 
\end{align*}
Since 
$\left( \gamma^{(p-1)m_1} - \gamma^{(p-1)m_2} \right)^{\frac{p^n-1}{p-1}}$ is a $(p-1)$-th root of unity and $p^i-1$ is divisible by $p-1$, we obtain $1 = -1$, which is absurd when $p$ is an odd prime. 
\end{proof}

\section{Relation to generalized almost bent functions}\label{sect:AB}

For a function $f \colon F \rightarrow F$, 
we define the \textbf{$p^n$-Walsh coefficients} of $f$ as follows: 
\begin{align*}
W_f (a, b) \coloneqq {\cal F} (\varphi _a + f_b) \ \ (a \in F, \ b \in F^{\times}), 
\end{align*}
where $\varphi_a$ is the components of the identity mapping on $F$. 
Similarly to the case that $p=2$, we define generalized almost bent functions. 

\begin{definition}
$f \colon F \rightarrow F$ is a \textbf{generalized almost bent} (GAB) function if 
\begin{align*}
W_f (a, b) \in \left\{ 0, \ \pm p^{\frac{n+1}{2}} \right\} \ \ \ 
\mbox{for all $a \in F$ and $b \in F^{\times}$. }
\end{align*}
\end{definition}

Note that when $p=2$, GAB functions coincide with AB functions. 
We have a characterization of GAB functions, 
which is a generalization of AB's one introduced in \cite{DF2003codes}. 

\begin{proposition} \label{GAB equiv}
Let $S^{(m)}_{a, b}$ be the number of solutions of the system of equations 
\begin{align*}
\left\{ \begin{array}{l}
x_1 + x_2 + \cdots + x_m  = a, \\
f (x_1) + f (x_2) + \cdots + f (x_m) = b. 
\end{array}
\right.
\end{align*}
Then 
$f$ is a GAB function if and only if 
\begin{align*}
S^{(3)}_{a, b} = \left\{ \begin{array}{cl}
p^n - p & (f (a) \ne b), \\
(p+1) p^n - p & (f (a) = b)
\end{array}
\right. \mbox{ for any $a$, $b \in F$. }
\end{align*}
\end{proposition}

\begin{proof}
We first define $p^n \times p^n$ matrices 
$W^{(m)}$, $S^{(m)}$, $E$ and $J$ which are indexed by elements in $F \times F$. 
The $(a, b)$-components of these matrices 
are as follows: 
\begin{gather*}
W^{(m)}_{a, b} := W_f (a, b)^m, \ 
S^{(m)}_{a b} := S^{(m)}_{a, b}, \ 
E_{a b} := \left\{ \begin{array}{cc}
1 & (a, b) = (0, 0) , \\
0 & \mbox{otherwise}. 
\end{array}
\right., \ 
J_{a b} := 1. 
\end{gather*}
By definition, $f$ is a GAB function if and only if 
\begin{align} \label{eq}
W_f (a, b) ^3 - p^{n + 1} W_f (a, b) = 0 \ \ (a \in F, \ b \in F^{\times}). 
\end{align}
Since if $b = 0$, then 
$\displaystyle W_f (a, 0) = \sum_{x \in F} \zeta_p^{{\rm Tr} (a x)} 
= \left\{ 
\begin{array}{cc}
p^n & (a = 0), \\
0 & (a \ne 0), 
\end{array}
\right.$
the equations (\ref{eq}) are equivalent to 
\begin{align} \label{eq2}
W^{(3)} - p^{n+1} W^{(1)} = \left( p^{3n} - p^{2 n + 1} \right) E. 
\end{align}
For any $m \in {\mathbb N}$, we have 
\begin{align*}
W_f (a, b) ^m 
&= \left( \sum_{x \in F} \zeta_p^{{\rm Tr} (a x) + {\rm Tr} (b f (x))}  \right)^m 
= \sum_{x_1, \dots, x_m \in F} 
\zeta_p^{{\rm Tr} (a (x_1 + \cdots + x_m))} \zeta_p^{ {\rm Tr} (b (f (x_1) + \cdots + f (x_m)))} \\
&= \sum_{s, t \in F} S^{(m)}_{s, t} \zeta_p^{{\rm Tr} (a s)} \zeta_p^{{\rm Tr} (b t)} 
= \sum_{s, t \in F} X_{a s} S^{(m)}_{st} X_{t b} , 
\end{align*}
where $X = \left[ X_{ab} \right]$ is defined in the proof of Proposition \ref{char_of_GAPN}. 
Hence we obtain 
\begin{align*}
W^{(m)} = X S^{(m)} X \ \ \ (m \in {\mathbb N}). 
\end{align*}
On the other hand, we have 
$\displaystyle XJX = \left[ \sum_{s, t \in F} X_{a s} J_{st} X_{tb} \right]$ and 
\begin{align*}
\sum_{s, t \in F} X_{a s} J_{st} X_{tb} 
= \sum_{s, t \in F} \zeta_p^{{\rm Tr} (a s + b t)} 
= \left\{ 
\begin{array}{cc}
p^{2 n} & ((a, b) = (0, 0)), \\
0 & (\mbox{otherwise}). 
\end{array}
\right.
\end{align*}
Hence $XJX = p^{2 n} E$. 
Therefore we obtain 
\begin{align*}
W^{(3)} - p^{n+1} W^{(1)} - (p^{3n} - p^{2 n+1}) E 
= X \left( S^{(3)} - p^{n+1} S^{(1)} - (p^n - p) J \right) X
\end{align*}
Then $X$ is regular, since $X X^{*} = p^n I$. 
Therefore the equation (\ref{eq2}) is equivalent to 
\begin{align*} 
S^{(3)} = p^{n+1} S^{(1)} + (p^n - p) J , 
\end{align*}
that is, 
$\displaystyle 
S_{a, b}^{(3)} =\left\{ \begin{array}{cl}
p^n - p & (f (a) \ne b), \\
(p+1) p^n - p & (f (a) = b)
\end{array}
\right.$ for any $a$, $b \in F$ since we have clearly 
$S^{(1)}_{a, b} 
= \left\{ \begin{array}{cl}
0 & (f (a) \ne b), \\
1 & (f (a) = b). 
\end{array}
\right.$ 
\end{proof}

\subsection{The case that $p=3$}

In this subsection, we assume that $p=3$ and 
\begin{align} \label{odd function}
f (-x) = - f (x) \ \ \mbox{for any $x \in F = {\mathbb F}_{3^n}$}. 
\end{align}
Then we have $f (0) = 0$ clearly. 
We have the following theorem which is a partial generalization of a relation between APN functions and AB functions introduced in \cite{CV1995links}. 

\begin{theorem} \label{main}
Let $f  \colon F \rightarrow F$ be a function with (\ref{odd function}). 
Assume that $d^{\circ} (f) \leq 3$. 
If $f$ is a GAB function, then $f$ is a GAPN function of algebraic degree $3$. 
\end{theorem}

\begin{proof}
Let $f$ be a GAB function. Since $f (0) = 0$, the system of equations 
\begin{align} \label{key system}
\left\{ 
\begin{array}{c}
x_1 + x_2 + x_3 = 0, \\
f(x_1) + f (x_2) + f (x_3) = 0
\end{array}
\right.
\end{align}
has $(3 + 1) 3^n - 3 = 3 (3^n - 1) + 3^n$ solutions by Proposition \ref{GAB equiv}. 
Since for any $b \in F$, 
\begin{align*}
f (0) + f (b) + f (2b) = f (b) + f (-b) = f(b) - f (b) = 0, 
\end{align*}
the solutions of (\ref{key system}) are only trivial solutions, that is 
\begin{gather} 
\Set{ (0, b, 2b),  (b, 2b, 0), (2b, 0, b) | b \in F^{\times} }, \ 
\Set{ (x, x, x) | x \in F }. \label{trivial solutions}\end{gather}
Assume that $f$ is not a GAPN function. 
Then by Proposition \ref{lem:alg_criterion1}, 
$\tilde{D}_a f (x) = \tilde{D}_a f (0)$ has a nontrivial solution 
$x_0 \in F \setminus \Set{ 0, a, 2a }$ 
for some $a \in F^{\times}$. 
On the other hand, by (\ref{odd function}), we have 
$\tilde{D}_a f (0) = 0$. 
Hence $(x_0, x_0+a, x_0+2 a)$ is a solution of the system (\ref{key system}), but 
this solution is not contained in any set of (\ref{trivial solutions}), 
which is absurd. 
Therefore $f$ is a GAPN function, and we have $d^{\circ} (f) = 3$ 
by Proposition \ref{lem:alg_criterion2}. 
\end{proof}

\begin{remark} \label{bad points}
\begin{enumerate}[label=(\arabic*)]
\item When $p=2$, any AB function is APN by \cite{CV1995links}. 
However the assumption of Theorem \ref{main} is necessary. 
In fact, there exists a function $f$ on ${\mathbb F}_{3^n}$ such that it is a GAB function but not a GAPN function when $d^{\circ} (f) > 3$. 
For example, let $n = 5$ and ${\mathbb F}_{3^5} = {\mathbb F}_{3} (\alpha)$ with $\alpha ^5 + 2 \alpha + 1 = 0$. 
Then the function $f : {\mathbb F}_{3^5} \to {\mathbb F}_{3^5}$ defined by $f(x) = x^{17}$ 
is a GAB function by a simple computation. However, we have 
$\left\{ x \in {\mathbb F}_{3^5} \mid D_1 f (x)  = 0 \right\} = \left\{ 0, 1, 2 \right\}$ and 
\begin{gather*}
\left\{ x \in {\mathbb F}_{3^5} \mid D_1 f (x)  = \alpha^3 + 2 \alpha ^2 + \alpha + 1 \right\} 
= \left\{ 2 \alpha + j, 
\alpha ^4 + \alpha ^3 + j 
\mid j \in {\mathbb F}_{3} \right\} . 
\end{gather*}
Thus $\tilde{N}_f (1, 0) = 3$ and $\tilde{N}_f (1, \alpha^3 + 2 \alpha ^2 + \alpha + 1) = 6$. 
Hence $f$ is not a GAPN function, and $d^{\circ} (f) > 3$ by Proposition \ref{lem:alg_criterion1}. 

\item 
When $p=2$, any quadratic APN function on ${\mathbb F}_{2^n}$ is an AB function if $n$ is odd by  \cite{BCCL2006almost}. 
Unfortunately, this property is not generalized in our case, that is, 
there exists a function $f$ on ${\mathbb F}_{3^n}$ such that $f$ is a GAPN function of algebraic degree $3$ but not a GAB function. 
In particular, the converse of Theorem \ref{main} is not true. 
For example, the function $f : {\mathbb F}_{3^5} \to {\mathbb F}_{3^5}$ defined by 
$f(x) = x^{11}$ is a GAPN function of algebraic degree $3$ (see Corollary \ref{GGF2}). 
However by a simple computation, we can see that the set of all Walsh coefficients of $f$ is 
$\left\{ 0, -9, 18, \pm 27, -36, 45, - 54 \right\}$, 
and hence $f$ is not a GAB function. 
\end{enumerate}
\end{remark}

\section{Construction of dual arcs}\label{sect:DA}

Let $ V $ be a vector space over a finite field $ \mathbb{F}_{q} $. 
A collection $ \mathcal{S} $ of $ m $-dimensional subspaces of $ V $ is called an $ (m-1) $-\textbf{dimensional dual arc} over $ \mathbb{F}_{q} $ if the following conditions are satisfied: 
\begin{enumerate}[label=(\roman*)]
\item $ \dim (X \cap Y) = 1 $ for any different $ X, Y \in \mathcal{S} $. 
\item $ X \cap Y \cap Z = 0 $ for any three mutually different $ X,Y,Z \in \mathcal{S} $. 
\end{enumerate}
If $ |\mathcal{S}|=(q^{m}-q)/(q-1)+1 $ then $ \mathcal{S} $ is called an $ (m-1) $-\textbf{dimensional dual hyperoval}. 

Let $ f $ be a quadratic function on $ \mathbb{F}_{2^{n}} $. 
We regard $ \mathbb{F}_{2^{n}} $ as an $ n $-dimensional vector space over $ \mathbb{F}_{2} $.  For every $ a \in \mathbb{F}_{2^{n}} $ we define a set $ X_{f}(a) \subset \mathbb{F}_{2^{n}} \oplus \mathbb{F}_{2^{n}} $ by 
\begin{align*}
X_{f}(a) \coloneqq \Set{(x,B_{f}(x,a)) \mid x \in \mathbb{F}_{2^{n}}}, 
\end{align*}
where $ B_{f}(x,a)=f(x+a)+f(x)+f(a)+f(0) $. 
Since $ f $ is quadratic, the form $ B_{f} $ is bilinear and the map $ x \mapsto (x, B_{f}(x,a)) $ is a injective linear map.  
Hence $ X_{f}(a) $ is $ n $-dimensional subspace in $ \mathbb{F}_{2^{n}} \oplus \mathbb{F}_{2^{n}} $ for every $ a \in \mathbb{F}_{2^{n}} $.
Let $ \mathcal{S}_{f} $ denote the collection of subspaces $ X_{f}(a) $. 
Yoshiara characterized quadratic APN functions on $ \mathbb{F}_{2^{n}} $ as follows: 
\begin{theorem}[Yoshiara {\cite[Theorem 2.1]{Yoshiara2008dimensional}}]
Let $ f \colon \mathbb{F}_{2^{n}} \rightarrow \mathbb{F}_{2^{n}} $ be a quadratic function. 
Then $ f $ is APN if and only if $ \mathcal{S}_{f} $ is an $ (n-1) $-dimensional dual hyperoval. 
\end{theorem}

Although the bilinearlity of $ B_{f} $ is very useful, the form $ \tilde{B}_{f} $ is hardly bilinear for $ p \geq 3 $.
We may resolve this problem with some modification. 
Let $ \mu $ be a map from $ F^{\times}=\mathbb{F}_{p^{n}}^{\times} $ to the set of $ \mathbb{F}_{p} $-linear automorphisms on $ F $ 
and let $\mu_a$ denote the image of $a$ by $\mu$. 
Let $ \nu $ be a permutation on $ F $ fixing $ 0 $. 

For such maps $ \mu, \nu $ and a function $ f \colon F \rightarrow F $, we define 
\begin{align*}
\tilde{B}_{f,\mu,\nu}(x,a) \coloneqq 
\left\{ \begin{array}{cl}
(\mu_{a} \circ \tilde{B}_{f})(x,\nu(a)) & (a \neq 0), \\
0 & (a=0). 
\end{array} \right.
\end{align*}
Note that for any $ a \in F^{\times} $ we have $ \tilde{B}_{f,\mu,\nu}(x,a)=0 $ if and only if $ \tilde{B}_{f}(x,\nu(a))=0 $. 
Hence when $ d^{\circ}(f) \leq p $ we have that $ f $ is a GAPN function if and only if
\begin{align*}
\Set{x \in F | \tilde{B}_{f,\mu,\nu}(x,a)=0} = \mathbb{F}_{p}\nu(a) \text{ for any } a \in F^{\times}
\end{align*} by Proposition \ref{lem:alg_criterion2}.

\begin{proposition}
Let $ f(x)=x^{d} $ be a monomial function with $ d^{\circ}(f) \leq p $. 
Define maps $ \mu, \nu $ by $ \mu_{a}(x)=a^{d}x $ and $ \nu(a)=a^{-1} $. 
Then $ \tilde{B}_{f,\mu,\nu}(x,a) $ is bilinear. 
\end{proposition}
\begin{proof}
Since $ d^{\circ}(f)\leq p $ the form $ [f]^{p} $ is multilinear. 
Hence $ \tilde{B}_{f}(x,a) $ is linear in $ x $. 
Moreover $ \tilde{B}_{f}(x,a) $ is homogeneous of degree $ d $ as a polynomial in $ x $ and $ a $. 
Therefore 
\begin{align*}
\tilde{B}_{f}(x,a)=\sum_{i}c_{i}x^{p^{i}}a^{d-p^{i}}
\end{align*}
for some $ c_{i} \in \mathbb{F}_{p} $. 
Then
\begin{align*}
\tilde{B}_{f,\mu,\nu}(x,a)=(\mu_{a}\circ \tilde{B}_{f})(x,\nu(a))
=a^{d}\left(\sum_{i}c_{i}x^{p^{i}}a^{p^{i}-d}\right)
=\sum_{i}c_{i}(xa)^{p^{i}}, 
\end{align*}
which is bilinear. 
\end{proof}

For the generalized Gold functions, we have another choice of maps $ \mu,\nu $ such that $ \tilde{B}_{f, \mu, \nu} $ is bilinear. 
\begin{proposition}
Let $ f(x)=x^{p^{i}+p-1} $ be the generalized Gold function. 
Define maps $ \mu,\nu $ by $ \mu_{a}(x)=a^{2-p}x $ and $ \nu(a)=a $.
Then $ \tilde{B}_{f,\mu,\nu}(x,a) $ is bilinear. 
\end{proposition}
\begin{proof}
By (ii) in Proposition \ref{GGF}, we have 
$
\tilde{B}_f (x, a) = - a^{p-1} x^{p^i} + a^{p^i + p-2} x
$. 
Hence we get 
\begin{align*}
\tilde{B}_f (x, a) =(\mu_{a}\circ \tilde{B}_{f})(x,\nu(a)) 
= a^{2-p} \left( - a^{p-1} x^{p^i} + a^{p^i + p-2} x \right) 
= - a x^{p^i} + a^{p^i} x, 
\end{align*}
which is bilinear. 
\end{proof}

\begin{proposition}\label{prop:mutually different three elements}
Let $ f $ be a GAPN function with $ d^{\circ}(f)=p $ and $ \mu,\nu $ as above.  Suppose that $ \tilde{B}_{f,\mu,\nu} $ is bilinear. 
Then the following hold: 
\begin{enumerate}[label=(\arabic*)]
\item $ \mathbb{F}_{p}\nu(a) = \mathbb{F}_{p}\nu(ia) $ for any $a \in F$ and $i \in \mathbb{F}_{p}^{\times} $. 
\item Three mutually different elements $ a,b,c \in F $ lie on the same line if and only if $ \nu(a-b) $ and $ \nu(a-c) $ are linearly dependent. 
\end{enumerate}
\end{proposition}
\begin{proof}
(1) Since $ \tilde{B}_{f,\mu,\nu} $ is bilinear, we have that $ \tilde{B}_{f,\mu,\nu}(x,a)=0 $ if and only if $ \tilde{B}_{f,\mu,\nu}(x,ia) $ for any $a \in F$ and $i \in \mathbb{F}_{p}^{\times} $. 
Hence 
\begin{align*}
\mathbb{F}_{p}\nu(a) 
= \Set{x \in F | \tilde{B}_{f,\mu,\nu}(x,a)=0}
= \Set{x \in F | \tilde{B}_{f,\mu,\nu}(x,ia)=0}
= \mathbb{F}_{p} \nu(ia). 
\end{align*}

(2) Suppose that mutually different elements $ a,b,c \in F $ lie on the same line. 
Then there exists $ i \in \mathbb{F}_{p}^{\times} $ such that $ a-b=i(a-c) $. 
We have $ \nu(a-b)=\nu(i(a-c)) $. 
By (1), there exists $ j \in \mathbb{F}_{p}^{\times} $ such that $ \nu(i(a-c))=j\nu(a-c) $. 
Hence we have $ \nu(a-b)=j\nu(a-c) $. 
Thus $ \nu(a-b) $ and $ \nu(a-c) $ are linearly dependent. 
The converse is similar. 
\end{proof}

Let $ f $ be a GAPN function with $ d^{\circ}(f)=p $ and $ \mu,\nu $ as above.  Suppose that $ \tilde{B}_{f,\mu,\nu} $ is bilinear.
For any $ a \in F $, we define 
\begin{align*}
X_{f,\mu,\nu}(a) \coloneqq \Set{(x,\tilde{B}_{f,\mu,\nu}(x,a)) | x \in F} \subset F \oplus F. 
\end{align*}
The bilinearity of $ \tilde{B}_{f,\mu,\nu} $ implies that $ X_{f,\mu,\nu}(a) $ is an 
$ n $-dimensional subspace in $ F \oplus F $. 
Let $ M \subset F $ be a set in which three mutually different elements do not lie on the same line. 
Let $ \mathcal{S}_{f,\mu,\nu,M} $ denote the collection of subspaces $ X_{f,\mu,\nu}(a) $, where $ a \in M $. 
\begin{theorem}
Suppose that $ n \geq 2 $. 
Then the collection $ \mathcal{S}_{f,\mu,\nu,M} $ is an $ (n-1) $-dimensional dual arc. 
\end{theorem}
\begin{proof}
Let $ a,b \in M $ be different elements. 
Suppose that $ (x,y) \in X_{f,\mu,\nu}(a) \cap X_{f,\mu,\nu}(b) $. 
Then we have $ y = \tilde{B}_{f,\mu,\nu}(x,a) = \tilde{B}_{f,\mu,\nu}(x,b) $. 
Hence $ \tilde{B}_{f,\mu,\nu}(x,a-b)=0 $. 
Therefore $ x \in \mathbb{F}_{p}\nu(a-b) $, 
and hence $ \dim (X_{f,\mu,\nu}(a) \cap X_{f,\mu,\nu}(b))=1 $. 
Since $n \geq 2$, $ X_{f,\mu,\nu}(a) $ is different from $ X_{f,\mu,\nu}(b) $. 

Next we suppose that $ a,b,c $ are mutually different elements in $ M $. 
Then by the above argument, 
$ X_{f,\mu,\nu}(a) $, $ X_{f,\mu,\nu}(b) $, $ X_{f,\mu,\nu}(c) $ 
are mutually different subspaces. 
On the other hand, since $a$, $b$, $c$ do not lie on the same line, 
$ \nu(a-b) $ and $ \nu(a-c) $ are linearly independent 
by Proposition \ref{prop:mutually different three elements}. 
Therefore 
\begin{align*}
X_{f,\mu,\nu}(a) \cap X_{f,\mu,\nu}(b) \cap X_{f,\mu,\nu}(c) \subset \mathbb{F}_{p} \nu(a-b) \cap \mathbb{F}_{p} \nu(a-c) = 0. 
\end{align*}
Hence $ \mathcal{S}_{f,\mu,\nu,M} $ is a dual arc. 
\end{proof}

\bibliographystyle{amsplain1}
\bibliography{bibfile}

\vspace{5mm}
\begin{table}[h]
\centering
\begin{tabular}{lll}
Masamichi Kuroda & \hspace*{10mm}& Shuhei Tsujie \\
Department of Mathematics & & Department of Mathematics \\
Hokkaido University & & Hokkaido University \\
Sapporo 060-0810 & & Sapporo 060-0810\\
Japan & & Japan  \\
m-kuroda@math.sci.hokudai.ac.jp & &  tsujie@math.sci.hokudai.ac.jp
\end{tabular}
\end{table}

\end{document}